\documentclass[10pt,a4paper]{amsart}
\usepackage[pdftex]{graphicx} 
\usepackage[usenames,dvipsnames]{xcolor}

\usepackage[english]{babel} \usepackage[utf8]{inputenc}

\usepackage[centering,margin=2.5cm]{geometry}
\usepackage{upgreek}
\usepackage[font=small,labelfont=bf]{caption}
\usepackage{amsfonts}
\usepackage{tikz}
\usepackage{tensor} 
\usepackage{lmodern}
\usepackage{verbatim}
\usepackage{amsmath}
\usepackage[all, cmtip]{xy}\usepackage{amssymb}
\usepackage{amsthm}
\usepackage{extpfeil}
\usepackage{bbm}
\usepackage{paralist}
\usepackage[colorlinks,citecolor=magenta,pagebackref=true,urlcolor=magenta,pdftex]{hyperref}
\usepackage{cancel}

\usepackage{soul}

\usepackage[T1]{fontenc}

\usepackage{nicefrac}
\usepackage{microtype}
\usepackage{mathrsfs}

\theoremstyle{plain}
\newtheorem{thm}{Theorem}
\newtheorem*{thm*}{Theorem}

\newtheorem{prop}[thm]{Proposition}
\newtheorem{lem}[thm]{Lemma}
\newtheorem{cor}[thm]{Corollary}

\newtheorem*{prb*}{Problem}

\theoremstyle{definition}
\newtheorem{rem}[thm]{Remark}
\newtheorem{df}[thm]{Definition}

\theoremstyle{remark}

\newtheorem{ex}[thm]{Example}

\newcommand\Defn[1]{\textbf{#1}}
\newcommand{\cm}[1]{}

\newcommand\mr[1]{\mathrm{#1}}

\newcommand\wt[1]{\widetilde{#1}}

\renewcommand{\k}{\mathbbm{k}}

\newcommand{\dps}{\mathrm{depth}\,}

\newcommand{\matteo}[1]{{\color{OrangeRed} \sf $\clubsuit\clubsuit\clubsuit$ Matteo: [#1]}}
\newcommand{\afshin}[1]{{\color{Green} \sf $\clubsuit\clubsuit\clubsuit$ Afshin: [#1]}}

\newcommand*\longhookrightarrow{\ensuremath{\lhook\joinrel\relbar\joinrel\rightarrow}}

\title{Connectivity of pseudomanifold graphs from an algebraic point of view}
\author{Karim A.~Adiprasito}
\author{Afshin Goodarzi}
\author{Matteo Varbaro}
\address{Institut des Hautes \'Etudes Scientifiques, Bures-sur-Yvette, France, \emph{and}
Einstein Institute for Mathematics, Hebrew University of Jerusalem, Jerusalem, Israel}
\email{adiprasito@ihes.fr, adiprssito@math.fu-berlin.de}
\address{Department of Mathematics, Kungliga Tekniska H\"ogskolan,  S-100 44 Stockholm, Sweden}
\email{afshingo@math.kth.se}
\address{Dipartimento di Matematica, Universit\`a di Genova Via Dodecaneso 35-16146, Genova, Italy}
\email{varbaro@dima.unige.it}

\date{\today}
\thanks{K.~Adiprasito was supported by an EPDI/IPDE postdoctoral fellowship
and a Minerva fellowship of the Max Planck Society.}
\thanks{M.~Varbaro was supported by PRIN  2010S47ARA\_003 ``Geometria delle Variet\`a Algebriche".}

\begin{document}

\begin{abstract}
The connectivity of graphs of simplicial and polytopal complexes is a classical subject going back at least to Steinitz, and the topic has since been studied by many authors, including Balinski, Barnette, Athanasiadis and Bj\"orner. In this note, we provide a unifying approach which allows us to obtain more general results. Moreover, we provide a relation to commutative algebra by relating connectivity problems to graded Betti numbers of the associated Stanley--Reisner rings. 
\end{abstract}
\maketitle


\section{Connectivity of the underlying graph}
Let $\Delta$ be a finite simplicial complex on the vertex set $[n]:=\{1,\ldots ,n\}$. The \Defn{underlying graph} (or \Defn{$1$-skeleton}) $G_\Delta$ of $\Delta$ is the graph obtained by restricting $\Delta$ to faces of cardinality at most two. 

A graph $G$ is said to be \Defn{$k$-connected} if it has at least $k$ vertices and removing any subsets of vertices of cardinality less than $k$ results in a connected graph. The \Defn{(vertex-)connectivity} $\kappa_G$ of $G$ is the maximum number $k$ such that $G$ is $k$-connected.

The classical Steinitz's theorem \cite{steinenc} asserts that a graph $G$ is the underlying graph of a $3$-polytope if and only if $G$ is $3$-connected and planar. In 1961, Balinski extended the ``only if'' direction of  Steinitz's theorem by showing that the underlying graph of a $d$-polytope is $d$-connected, cf.\ \cite{Z}. David Barnette showed that the same bound is also valid for the connectivity number of underlying graphs of $(d-1)$-dimensional pseudomanifolds \cite{Barnette2}.

Athanasiadis \cite{Ath} showed that if the pseudomanifold is also flag (i.e.\ the clique complex of its $1$-skeleton), then this lower bound can be improved to $2d-2$. Bj\"orner and Vorwerk quantified this connection using the notion of banner complexes \cite{BV}. \\

The purpose of this note is to provide a unifying approach which allows us to obtain more general results.

The proof is inspired by a relation of connectivity to the Hochster's formula (observed in \cite{Goodarzi}) from commutative algebra and simple estimates for the size of certain flag complexes \cite{ACMF}.


\section{Basics in commutative algebra}
We start by recalling some notions, and refer to \cite{MS, HerzogHibi} for exact definitions and more details.
Let $I$ be a graded ideal in the polynomial ring $S=\k[x_1,\ldots,x_n]$ in $n$ variables over a field $\k$. Let 
\[\mathbf{F}_{S/I}\ :=\  0\rightarrow F_p\rightarrow F_{p-1}\rightarrow\cdots\rightarrow F_1\rightarrow F_0\rightarrow S/I\rightarrow 0,
\]
be the minimal graded free resolution of $S/I$, with $F_i=\bigoplus_j S(-j)^{b_{i,j}}$ in homological degree $i$. The number $b_{i,j}=b_{i,j}(S/I)$ is the \Defn{graded Betti number} of $S/I$ in \Defn{homological degree} $i$ and \Defn{internal degree} $j$. 
The length of the $j$-th row in the Betti table will be denoted by $\mathit{lp}_j(S/I)$, that is 
\[\mathit{lp}_j(S/I)\ :=\ \max\{i\mid b_{i,i+j-1}(S/I)\neq 0\}.\]
We also denote by $t_i(S/I)$ the \Defn{maximum internal degree} of a minimal generator in the homological degree $i$ that is $\max\{j\mid b_{i,j}\neq 0\}$. The \Defn{projective dimension} of $S/I$ is the maximum $i$ such that $b_{i,j}\neq 0$, for some $j$. The \Defn{regularity} of $S/I$ is defined to be $\max_i \{t_i(S/I)-i\}$. 

\section{Connectivity via graded Betti numbers}

Let $\Delta$ be a simplicial complex on the vertex set $[n]$. The \Defn{Stanley--Reisner ideal} $I_\Delta\subset S$ of $\Delta$ is the ideal generated by monomials $\mathbf{x}_F:=\prod_{i\in F} x_i$ for all $F$ not in $\Delta$. The quotient ring $\k[\Delta]=S/I_\Delta$ is called the \Defn{face ring} of $\Delta$. 
In this case, \Defn{Hochster's formula} provides an interpretation of the graded Betti numbers in terms of the reduced homology of induced sub-complexes of $\Delta$. More precisely, it asserts that 
\begin{equation*}\label{Hochster}
b_{i,j}(\k[\Delta])\ =\ \sum_{\# W=j} \dim_{\k}\wt{H}_{j-i-1} (\Delta_W).
\end{equation*}


In \cite[Theorem 3.1]{Goodarzi} it was observed that the graded Betti numbers computed in Hochster's formula are naturally connected to connectivity of the underlying graphs.
\begin{prop}\label{Connectivity-LStrand}
Let $\Delta$ be a simplicial complex on the vertex set $[n]$ and $\kappa$ be the connectivity number of its underlying graph. Then one has \[\kappa+\mathit{lp}_2(\k[\Delta])=n-1.\]
\end{prop}

\begin{proof}
It suffices to observe that $\kappa = n- \mathrm{max}\{\#W : \widetilde{H}_0(X_W) \neq 0\}=n-1-\mathit{lp}_2(\k[\Delta])$.

Alternatively (and algebraically) it suffices to observe that $\Delta$ and $\mr{Cl}(G)$, the clique complex of the underlying graph $G$ of $\Delta$, both have the same connectivity. Hence, it suffices to verify the result in the case of flag complexes and we may assume that $\Delta=\mr{Cl}(G)$. The result follows from ~\cite[Theorem~3.1]{Goodarzi}.
\end{proof}

Before presenting our next result, we shall introduce two properties. 
\begin{df} Let  $I$ be a graded ideal such that $S/I$ is of regularity $r$. Set $m= \mathit{lp}_r(S/I)$. We say $S/I$ satisfies the \Defn{property $\mathfrak{A}$} if 
\begin{compactenum}[(1)]
\item $\mathit{lp}_2(S/I)\leq m$,
\item $b_{m-i,m-i+1}(S/I)\leq b_{i,i+r-1}(S/I)$.
\end{compactenum}
We also say that $S/I$ satisfies the \Defn{property $\mathfrak{B}_s$} if for all $i<s$ one has $t_i(S/I)<r+i-1$. 
\end{df}

\begin{rem} \textbf{Relations to Poincar\'e duality and the Koszul property}
\begin{compactenum}[(1)] 
\item If $S/I$ is Gorenstein, then it satisfies the property $\mathfrak{A}$. However, the property only requires a much simpler property than Poincar\'{e}--Lefschetz duality; a simple inequality shall be enough, see Lemma \ref{Poincare}.
\item If $I$ is generated by quadratic monomials, then it is easy to see that it $t_s(S/I)\leq 2s$ for all~$s$ and therefore $S/I$ satisfies $\mathfrak{B}_{r-1}$. This fact is valid more generally when $S/I$ is Koszul as was shown by Backelin in \cite{backelin}, see also Kempf \cite{Kempf}. 
\end{compactenum}
\end{rem}
\begin{prop}\label{Diff}
Let $I$ be a graded ideal in polynomial ring $S$. Moreover, assume that $S/I$ has regularity $r$ and satisfies the properties $\mathfrak{A}$ and $\mathfrak{B}_s$. Then one has
\[ s\leq \mathit{lp}_r(S/I)-\mathit{lp}_2(S/I).
\]\end{prop}
\begin{proof}
We have
\begin{align*}
  \mathit{lp}_r(S/I)-\mathit{lp}_2(S/I)\ =&\ m-\max\{j\mid b_{j,j+1}(S/I)\neq 0\} &\\
\ =&\ \min\{m-j\mid b_{j,j+1}(S/I)\neq 0\} &\\
\ \geq&\ \min\{m-j\mid b_{m-j,m-j+r-1}(S/I)\neq 0\} &\text{(Property $\mathfrak{A}$)} \\
\ =&\ \min\{k\mid b_{k,k+r-1}(S/I)\neq 0\}&\\
\ \geq&\ \min\{k\mid t_k(S/I)\geq k+r-1\}&
\end{align*}
where the last term is at least $s$ by Property $\mathfrak{B}_s$. 
\end{proof}
By Bakelin's result the regularity of a Koszul ring is bounded above by its projective dimension. As an immediate consequence of the previous result, we get the following tight bound for Gorenstein Koszul rings.
\begin{cor}
The regularity of a Gorenstein Koszul ring $S/I$ is at most $\mathit{projdim}(S/I)-\mathit{lp}_2(S/I)+1.$ \qed
\end{cor}
\begin{thm}\label{main}
Let $\Delta$ be a $(d-1)$-dimensional simplicial complex with nontrivial top-homology. Also, assume that $\k[\Delta]$ satisfies the properties $\mathfrak{A}$ and $\mathfrak{B}_s$. Then the underlying graph is $(d+s-1)$-connected.
\end{thm}
\begin{proof}
Note that the regularity of $\k[\Delta]$ is equal to $d$ since $\Delta$ has nontrivial top-homology. So, it follows from Proposition~\ref{Diff} that
\[ s\leq \mathit{lp}_d(\k[\Delta])-\mathit{lp}_2(\k[\Delta]).
\]
Note that $\mathit{lp}_d(\k[\Delta])=n-d$. So, by Corollary~\ref{Connectivity-LStrand} we get

\[ s\leq n-d -(n-\kappa_\Delta-1),
\]
where $\kappa_\Delta$ stands for the connectivity number of the underlying graph of $\Delta$. Therefore
\[ \kappa_\Delta\geq d+s-1. \qedhere
\]

\end{proof}
\begin{rem}
As a special case of Theorem~\ref{main}, we can consider $\Delta$ to be Gorenstein*. Then $\k[\Delta]$ satisfies the property $\mathfrak{B}_1$. Moreover, if $\Delta$ is also flag, then it satisfies the property $\mathfrak{B}_{d-1}$, since $t_i(\k[\Delta])\leq 2i$ for any $i$ and $\k[\Delta]$ is $d$-regular. 
\end{rem}

\section{A Poincar\'e--Lefschetz-type inequality for minimal cycles}

Recall that a \Defn{minimal $d$-cycle} $\Sigma$ (w.r.t.\ a coefficient ring $R$) is a pure $d$-dimensional complex that supports precisely one homology $d$-class $\zeta$ whose support is the complex itself.
For instance, every pseudomanifold is a minimal cycle (over $\mathbb{Z}/2\mathbb{Z}$); and so is every triangulation of a closed, connected manifold. 

\begin{lem}\label{Poincare}
Let $\Sigma$ denote any minimal $d$-cycle and $W$ a subset of the vertex-set $\mr{V}(\Sigma)$. Then
\[\mr{rk}\ \widetilde{H}_0(\Sigma_{W})\ \le\ \mr{rk}\ \widetilde{H}_{d-1}(\Sigma_{V(\Sigma)\setminus W})\]
\end{lem}

\begin{proof}
Since $\Sigma$ supports a global $d$-cycle (by minimality), we have an injection
\[{H}^0(\Sigma_{W})\ \longhookrightarrow \ {H}_d(\Sigma,\Sigma\setminus\Sigma_{W}).\]
To see this, notice that the restriction of the global $d$-cycle to any connected component of $\Sigma_{W}$ induces a relative cycle for $(\Sigma,\Sigma\setminus\Sigma_{W})$.

Now, since $\Sigma_{V(\Sigma)\setminus W}$ is homotopically equivalent to $\Sigma\setminus\Sigma_{W}$, the exact sequence
\[0\ \longrightarrow\ \widetilde{H}_d(\Sigma)\ \longrightarrow\ \widetilde{H}_d(\Sigma,\Sigma_{V(\Sigma)\setminus W})\ \longrightarrow\ \widetilde{H}_{d-1}(\Sigma_{V(\Sigma)\setminus W}) \ \longrightarrow\ \cdots,\]
implies 
\begin{align*}
&\mr{rk}\  \widetilde{H}_{d-1}(\Sigma_{V(\Sigma)\setminus W})+1\\
=\ &\mr{rk}\  \widetilde{H}_{d-1}(\Sigma_{V(\Sigma)\setminus W})+\mr{rk}\  \widetilde{H}_d(\Sigma)\\
\ge\ & \mr{rk}\  \widetilde{H}_d(\Sigma,\Sigma_{V(\Sigma)\setminus W})\\
\ge\ & \mr{rk}\  {H}^0(\Sigma_{W}). \qedhere
\end{align*}

%
%
%
\end{proof}

\section{Applications to connectivity of manifolds}

Let $\Delta$ be a $(d-1)$-dimensional simplicial complex
on the vertex set $\mr{V}(\Delta)$. Recall the notion of banner complexes of \cite{BV}:
\begin{compactitem}[$\circ$]
\item A subset $W$ of $\mr{V}(\Delta)$ is called \Defn{complete} if every
two vertices of $W$ form an edge of $\Delta$.
\item A complete set  $W \subseteq \mr{V}(\Delta)$ is \Defn{critical}
if $W \setminus \{v\}$ is a face of $\Delta$ for some $v \in W$.
\item We say that $\Delta$ is \Defn{banner}
if every critical complete set $W$ of size at least $d$ is a face of $\Delta$.
\item We define the \Defn{banner number} of $\Delta$ to be

\begin{eqnarray*}
\mr{b}(\Delta)&=&\min\left\{b \quad :
	\begin{array}{ll}
		\mr{lk}_\sigma \Delta \mbox{ is banner or the boundary of the $2$-simplex} \\
		\mbox{for all faces } \sigma\in\Delta \mbox{ of cardinality $b$ and degree $d$}
	\end{array}
\right\},
\end{eqnarray*}
where the \Defn{degree} of a face is the maximal cardinality of a facet containing it. 
\end{compactitem}

\noindent Note that our notions of banner complexes and banner numbers are slightly more general then the ones introduced in~\cite{BV}. However, if the complex is pure the definitions coincide.

\begin{lem}\label{banner}
Let $\Delta$ be a $(d-1)$-dimensional simplicial complex.
\begin{compactenum}[\rm(a)]
\item\label{a} If  $\sigma$ is a face of degree $d$ in $\Delta$, then $\mr{b}(\mr{lk}_\sigma \Delta)\le \max\{0,\mr{b}(\Delta)- \# \sigma\}$.
\item\label{b} If $\Delta$ has nontrivial top-homology and $\mr{b}(\Delta)<d-2$, then every induced subcomplex of $\Delta$ having nontrivial $(d-2)$-homology has at least $2d-2-\mr{b}(\Delta)$ vertices.
\end{compactenum}
\end{lem}

\begin{proof}
The part~\eqref{a} is clear from the definition. For claim~\eqref{b}, let us first show that, if $\Delta$ is banner, then every induced subcomplex $\Gamma$ of $\Delta$ such that $\widetilde{H}_{d-2}(\Gamma)\neq 0$ has at least $2d-2$ vertices by induction on $d$. If $d=3$, this is clear because $\Delta$ is flag.

Let $d>3$. We may assume that no induced subcomplex of $\Delta$ has a nontrivial $(d-1)$-dimensional cycle: indeed, such a subcomplex is forced to have dimension $d-1$, so it would be banner and we could replace $\Delta$ with it. Furthermore, we may assume that $\Gamma$ is a minimal induced subcomplex with the property that $\widetilde{H}_{d-2}(\Gamma)\neq 0$. Under such a minimality assumption, the link of any vertex of $\Gamma$ admits a nontrivial homology cycle in dimension $d-3$. Take a vertex $v$ of $\Gamma$. Since $\mr{lk}_v \Gamma$ is an induced subcomplex of $\mr{lk}_v \Delta$, which is banner and admits a nontrivial $(d-2)$-cycle, by induction $\mr{lk}_v\Gamma$ has at least $2d-4$ vertices. Moreover, $v$ cannot be a cone point of $\Gamma$ because $\widetilde{H}_{d-2}(\Gamma)\neq 0$, so $\# {\mr V}(\Gamma)\geq 2d-2$. 

The claim~\eqref{b} now follows by induction on the banner number and claim~\eqref{a}. 
\end{proof}

\begin{rem}
While a flag simplicial complex (not necessarily of dimension $d-1$) supporting a nontrivial $(d-1)$-cycle has at least $2d$ vertices, this is false for banner complexes. Take the boundary of a $d$-simplex, and join one facet with an external edge: the resulting complex is a $(d+1)$-dimensional banner complex supporting a nontrivial $(d-1)$-cycle, but with only $d+3$ vertices.
\end{rem}

\begin{lem}\label{B}
Let $\Delta$ be a pure $(d-1)$-dimensional complex with nontrivial top-homology. If $\mr{b}(\Delta)<d-2$, then $\k[\Delta]$ satisfies the property $\mathfrak{B}_{d-\mr{b}(\Delta)-1}$. 
\end{lem}
\begin{proof}
Notice that the regularity of $\k[\Delta]$ is $d$, since $\Delta$ has a nontrivial top-homology. If $b_{i,i+d-1}(\k[\Delta])\neq 0$, by Hochster's formula there exists a subset $W\subseteq V(\Delta)$ of cardinality $i+d-1$ such that $\Delta_W$ supports a nontrivial $(d-2)$-cycle. By part~(b) of Lemma~\ref{banner}, thus:
\[i\geq d-\mr{b}(\Delta)-1 . \qedhere\]
\end{proof}

\begin{thm}\label{thm:bannerconnect}
Let $\Delta$ be an $(d-1)$-dimensional minimal cycle. Then the underlying graph of $\Delta$ is $(2d-\mr{b}(\Delta)-2)$-connected.
\end{thm}
\begin{proof}
If $\mr{b}(\Delta)=d-2$, then it is easy to see that $\k[\Delta]$ satisfies $\mathfrak{B}_1$. By Lemma~\ref{Poincare} and Hochster's formula, $\k[\Delta]$ satisfies also the property $\mathfrak{A}$. Therefore, the result follows from Theorem~\ref{main}.

If $\mr{b}(\Delta)<d-2$, by Lemmata~\ref{Poincare} (together with Hochster's formula) and~\ref{B}, $\k[\Delta]$ satisfies the properties $\mathfrak{A}$ and $\mathfrak{B}_{d-\mr{b}(\Delta)-1}$. Therefore, the result follows from Theorem~\ref{main}.
\end{proof}
\begin{cor}
Let $\Delta$ be a flag (or more generally banner) $(d-1)$-dimensional minimal cycle. Then the underlying graph of $\Delta$ is $(2d-2)$-connected.
\end{cor}
\begin{proof}
If $\Delta$ is a banner complex, then $b(\Delta)=0$. 
\end{proof}
%

{\small
\bibliographystyle{myamsalpha}
\bibliography{References}}

\providecommand{\noopsort}[1]{}
\providecommand{\bysame}{\leavevmode\hbox to3em{\hrulefill}\thinspace}
\providecommand{\MR}{\relax\ifhmode\unskip\space\fi MR }
\providecommand{\MRhref}[2]{%
  \href{http://www.ams.org/mathscinet-getitem?mr=#1}{#2}
}
\providecommand{\href}[2]{#2}
\begin{thebibliography}{{Kem}90}

\bibitem[ANT14]{ACMF}
K.~A. Adiprasito, E.~Nevo, and M.~Tancet, \emph{Improved bounds for
  cohomological dimension and betti numbers of flag complexes}, in preparation.

\bibitem[{Ath}11]{Ath}
C.~A. {Athanasiadis}, \emph{{Some combinatorial properties of flag simplicial
  pseudomanifolds and spheres}}, {Ark. Mat.} \textbf{49} (2011), no.~1, 17--29
  (English).

\bibitem[Bac88]{backelin}
J.~Backelin, \emph{Relations between rates of growth of homologies}, Research
  Reports in Mathematics, Mat. institutionen, Stockholms univ. \textbf{25}
  (1988).

\bibitem[Bar82]{Barnette2}
D.~Barnette, \emph{Decompositions of homology manifolds and their graphs},
  Israel J. Math. \textbf{41} (1982), no.~3, 203--212.

\bibitem[BV14]{BV}
A.~Bj{\"o}rner and K.~Vorwerk, \emph{On the connectivity of manifold graphs},
  {Proc. Am. Math. Soc.} (2014) (English), to appear.

\bibitem[Goo14]{Goodarzi}
A.~Goodarzi, \emph{Clique vectors of $k$-connected chordal graphs}, preprint,
  8~pages, \href{http://arxiv.org/abs/1403.6210}{arXiv:1403.6210}, 2014.

\bibitem[HH11]{HerzogHibi}
J.~Herzog and T.~Hibi, \emph{Monomial ideals}, Graduate Texts in Mathematics,
  vol. 260, Springer-Verlag London Ltd., London, 2011.

\bibitem[{Kem}90]{Kempf}
G.~R. {Kempf}, \emph{{Some wonderful rings on algebraic geometry}}, {J.
  Algebra} \textbf{134} (1990), no.~1, 222--224.

\bibitem[MS05]{MS}
E.~Miller and B.~Sturmfels, \emph{Combinatorial commutative algebra}, Graduate
  Texts in Mathematics, vol. 227, Springer-Verlag, New York, 2005.

\bibitem[Ste22]{steinenc}
E.~Steinitz, \emph{{P}olyeder und {R}aumeinteilungen}, Encyklop\"adie der
  mathematischen Wis\-sen\-schaf\-ten, 
  Dritter Band: Geometrie, III.1.2., Heft~9, Kapitel \mbox{III\,A\,B\,12}
  (W.~Fr. Meyer and H.~Mohrmann, eds.), B.~G.~Teubner, Leipzig, 1922,
  pp.~1--139.

\bibitem[Zie95]{Z}
G.~M. Ziegler, \emph{Lectures on {P}olytopes}, Graduate Texts in Mathematics,
  vol. 152, Springer, New York, 1995, Revised edition, 1998; seventh updated
  printing 2007.

\end{thebibliography}

\end{document}